\documentclass{amsart}[12pt, titlepage, twoside, reqno]
\usepackage[toc,page]{appendix}
\usepackage{amsmath}
\usepackage{amsfonts}
\usepackage{amsthm}
\usepackage{tikz}
\usepackage[inline]{asymptote}
\usepackage{asymptote}
\usepackage{graphicx}
\usepackage{listings}
\usepackage{amssymb}
\usepackage[T1]{fontenc}
\usepackage[normalem]{ulem}
\usepackage{caption}
\usepackage[section]{algorithm}
\usepackage{algpseudocode}
\usepackage{fancyhdr}
\usepackage[tableposition=top]{caption}
\usepackage{tabularx}
\usepackage{multirow}
\usepackage{verbatim}
\newcounter{nalg}[section] 
\renewcommand{\thenalg}{\thesection .\arabic{nalg}} 
\DeclareCaptionLabelFormat{algocaption}{\textbf{Algorithm \thenalg}} 
\newtheorem{ex}{Example}[section]
\newtheorem{prop}[ex]{Proposition}
\newtheorem{lem}[ex]{Lemma}
\newtheorem{thm}[ex]{Theorem}

\numberwithin{equation}{section}
\usepackage{footmisc}

\theoremstyle{definition}
\newtheorem{definition}[ex]{Definition}

\DeclareMathOperator*{\BigExpt}{\raisebox{-2.5pt}{\mbox{\LARGE $\mathbb{E}$}}}
\usepackage[hidelinks, colorlinks=true]{hyperref} 

\usepackage{caption}
\usepackage[section]{algorithm}
\usepackage{algpseudocode}
\usepackage{mathtools}

\DeclarePairedDelimiter\ceil{\lceil}{\rceil}
\DeclarePairedDelimiter\floor{\lfloor}{\rfloor}
\renewcommand{\thenalg}{\thesection .\arabic{nalg}} 
\DeclareCaptionLabelFormat{algocaption}{\textbf{Algorithm \thenalg}} 
\newtheorem{alg}[ex]{Algorithm}
\newtheorem{corollary}[ex]{Corollary}
\usepackage{ltablex}

\begin{document}

\title{Improved Lower Bounds for Kissing Numbers in Dimensions 25 through 31}
\author{Kenz Kallal, Tomoka Kan, and Eric Wang}
\date{\today}
\maketitle
\begin{abstract}
The best previous lower bounds for kissing numbers in dimensions 25--31 were constructed using a set $S$ with $|S| = 480$ of minimal vectors of the Leech Lattice, $\Lambda_{24}$, such that $\langle x, y \rangle \leq 1$ for any distinct $x, y \in S$. Then, a probabilistic argument based on applying automorphisms of $\Lambda_{24}$ gives more disjoint sets $S_i$ of minimal vectors of $\Lambda_{24}$ with the same property. Cohn, Jiao, Kumar, and Torquato proved that these subsets give kissing configurations in dimensions 25--31 of given size linear in the sizes of the subsets. We achieve $|S| = 488$ by applying simulated annealing. We also improve the aforementioned probabilistic argument in the general case. Finally, we greedily construct even larger $S_i$'s given our $S$ of size $488$, giving increased lower bounds on kissing numbers in $\mathbb{R}^{25}$ through $\mathbb{R}^{31}$.
\end{abstract}
\section{Introduction}\label{intro}
\subsection{Kissing Numbers}
A \emph{kissing configuration} in $\mathbb{R}^n$ is a set of non-overlapping unit spheres externally tangent to the unit sphere $S^{n-1}$. The \emph{kissing number} in $\mathbb{R}^n$ is the maximal size of kissing configurations in $\mathbb{R}^n$. The kissing number is known only for dimensions 1, 2, 3, 4, 8 and 24 (the bold figures in Table~\ref{kissing}).
\begin{table}[h]
\centering
\begin{tabular}{c c | c c}
Dimension & Kissing Number & Dimension & Kissing Number \\
\hline
1 & \textbf{2} & 13 & 1154\\
2 & \textbf{6} & 14 & 1606\\
3 & \textbf{12} & 15 & 2564\\
4 & \textbf{24} & 16 & 4320\\
5 & 40 & 17 & 5346\\
6 & 72 & 18 & 7398\\
7 & 126 & 19 & 10668\\
8 & \textbf{240} & 20 & 17400\\
9 & 306 & 21 & 27720\\
10 & 500 & 22 & 49896\\
11 & 582 & 23 & 93150\\
12 & 840 & 24 & \textbf{196560}\\
\end{tabular}
\caption{The best lower bounds known for kissing numbers in dimensions 1--24.}
\label{kissing}
\end{table}
In dimensions 25 to 31, the best lower bounds previously known for kissing numbers were proved in \cite{Cohn}. We improve these bounds by using computer programs to maximize the size of an initial set of minimal vectors of the Leech lattice, and to construct subsequent mutually disjoint sets of minimal vectors. We also improve \cite{Cohn}'s probabilistic argument for the construction of these subsets, which improves the lower bounds, despite being surpassed by the computer-aided construction in specific cases. We summarize our results in Table~\ref{newresults}.
\begin{table}
\centering
\begin{tabular} {c | c  c  c}
Dimension  & New Bounds & Previous Bounds (2011)\\
\hline
25 & 197048 & 197040\\
26 & 198512 & 198480\\
27 & 199976 & 199912\\
28 & 204368 & 204188\\
29 & 208272 & 207930\\
30 & 219984 & 219008\\
31 & 232874 & 230872\\
\end{tabular}
\caption{New kissing number lower bounds in dimensions 25--31.}
\label{newresults}
\vspace{-8mm}
\end{table}
Although we set new records for the lower bounds, our constructions are not optimal. However, we introduce techniques which may help improve this approach further. 

Next, we introduce some basic concepts and the main setup in \cite{Cohn} to help explain our new results in Sections~\ref{anneal}--\ref{math}. Basic geometric considerations allow us to reformulate the kissing number in $\mathbb{R}^n$ as the maximum number $N$ of points $x_1, \dots ,x_N$ on the unit sphere $S^{n-1}$ such that $\langle x_i, x_j \rangle \leq 1/2$ for $i \neq j$. This reformulation makes it significantly easier to test large kissing configurations.

\subsection{The Kissing Number in 24 Dimensions}
The highest-dimensional Euclidean space in which the kissing number is known is $\mathbb{R}^{24}$, where the centers of the tangent spheres are the 196560 minimal vectors of the Leech lattice $\Lambda_{24}$ (see \cite{Conway}). The optimality of $\Lambda_{24}$ was proven in \cite{Levenstein} and \cite{Sloane}.

Let $\mathcal{C}$ denote the set of minimal vectors of $\Lambda_{24}$. Every element of $\mathcal{C}$ is of the ``shape'' $(\pm 4^2,0^{22})/\sqrt{8}$, $(\pm 2^8, 0^{16})/\sqrt{8}$, or $(\pm 3, \pm 1^{23})/\sqrt{8}$, where $(\pm a^u,\pm b^{w})$ represents vectors which contain $u$ copies of $\pm a$ and $w$ copies of $\pm b$ with signs independent of each other (see \cite{thompson}). Note that not every such vector is in $\mathcal{C}$. There are a total of $|\mathcal{C}| = 196560$ minimal vectors, all with norm $2$.
Let $S$ be a subset of $\mathcal{C}$ which satisfies $\langle x, y \rangle \leq 1$ for all distinct $x, y \in S$. We will see that maximizing $|S|$ will allow us to increase the lower bounds on kissing numbers in $\mathbb{R}^{25}$ through $\mathbb{R}^{31}$.

\subsection{Dimensions past 24}\label{sect1.3}
Given mutually disjoint subsets $S_i$ of $\mathcal{C}$ satisfying the same inner product condition as $S$, we construct configurations in $\mathbb{R}^{24 + d}$ that yield the lower bounds in Table~\ref{cohnsums}.
\begin{table}
\centering
\begin{tabular}{c | c}
Dimension & Lower Bounds on Kissing Number \\
\hline 
25 & $196560 + |S_1|$ \\
26 & $196560 + 2|S_1| + 2|S_2|$ \\
27 & $196560 + 2|S_1| + 2|S_2| + \sum_{i = 3}^5 |S_i|$ \\
28 & $196560 + 2\sum_{i = 1}^8 |S_i|$ \\
29 & $196560 + 2\sum_{i = 1}^8 |S_i| + \sum_{i = 9}^{16} |S_i|$ \\
30 & $196560 + 2\sum_{i = 1}^{24} |S_i|$ \\
31 & $196560 + 2\sum_{i = 1}^{24} |S_i| + \sum_{i = 25}^{51} |S_i|$
\end{tabular}
\caption{Lower bounds on kissing numbers in $\mathbb{R}^n$}
\label{cohnsums}
\vspace{-8mm}
\end{table}
These are a special case of \cite[Theorem~7.5]{Cohn}:
\begin{thm}\label{thm1.2}
Let $S_1, \ldots, S_n$ be mutually disjoint subsets of $\mathcal{C}$ with $\langle x,y \rangle \leq 1$ for all distinct $x,y \in S_i$, where $i = 1,2, \dots, n$. Suppose that we have partitioned a kissing configuration in $S^{d-1}$ into disjoint subsets $T_1, \ldots, T_n$ such that for each $i$, $\langle x, y \rangle \leq -1/2$ for any $x, y \in T_i$. Then the kissing number in $\mathbb{R}^{24 + d}$ is at least \[196560+\sum_{i = 1}^{n} (|T_i| - 1)|S_i|.\]
\end{thm}
\begin{proof}[Sketch of proof]
We can verify that the following set of vectors of norm $4$ has pairwise dot product at most 2:
\[\left\{(x, 0) \in \mathbb{R}^{24} \times \mathbb{R}^d \mid x \in \mathcal{C} \backslash \bigcup_{i=1}^n S_i\right\} \cup \bigcup_{i=1}^n\left\{\left(x\sqrt{2/3}, y\sqrt{4/3}\right) \mid x \in S_i, y \in T_i\right\}.\]
\end{proof}
Note that $|T_i| \leq 3$, where equality holds when $T_i$ contains three vectors arranged in an equilateral triangle in the same plane as the origin. 

\cite{Cohn} gives the method indicated in Proposition~\ref{prop1.4} below for obtaining lower bounds on the $|S_i|$'s probabilistically, which we later improve in Section~\ref{math}. Their proof uses the automorphism group of $\Lambda_{24}$, defined below:
\begin{definition}[Automorphism group of $\Lambda_{24}$]\label{defn1.3}
The automorphisms of $\Lambda_{24}$ are the bijections $g : \Lambda_{24} \to \Lambda_{24}$ which preserve distance in $\mathbb{R}^{24}$ and fix the origin. These automorphisms form a group acting on $\Lambda_{24}$, known as the Conway group $\textrm{Co}_0$.  
\end{definition}

It follows from this definition that automorphisms also preserve inner product.


We present \cite[Lemma 7.4]{Cohn} in Proposition~\ref{prop1.4}. The proof of this lemma involves recursively bounding the $|S_i|$'s from Theorem~\ref{thm1.2} by translating $S$ by a random element $g \in \textrm{Co}_0$, and deleting the intersection of $gS$ with $S_1 \cup \dots \cup S_{i-1}$.

By using the bounds on the $|S_i|$'s given in Proposition~\ref{prop1.4}, the bounds in Table~\ref{cohnsums}, and an initial $S$ of size 480, \cite{Cohn} arrived at the numerical lower bounds on kissing numbers displayed in the rightmost column of Table~\ref{newresults}.

\begin{prop}\label{prop1.4}
\cite{Cohn} Let $S_1 = S$. Then for any $n \geq 1$, there exist mutually disjoint subsets $S_2, \ldots, S_n \subset \mathcal{C}$ such that $\langle x, y \rangle \leq 1$ for all $x, y \in S_i$, and 
\[|S_i| \geq |S|\left(1 - \frac{\sum_{j = 1}^{i-1}|S_j|}{|\mathcal{C}|}\right) = |S|\left(1 - \frac{\sum_{j = 1}^{i-1}|S_j|}{196560}\right).\]
Moreover, if $S$ is antipodal, each of the $S_i$'s will also be antipodal.
\end{prop}

Starting from an antipodal set $S$ with $|S| = 480$,  \cite{Cohn} applied Theorem~\ref{thm1.2} and Proposition~\ref{prop1.4} to give the lower bounds shown in the rightmost column of Table~\ref{newresults}.

These two results yield two clear methods to increase the lower bounds on kissing numbers in $\mathbb{R}^{25}$ through $\mathbb{R}^{31}$: constructing $S$ of size larger than 480, and constructing $S_i$'s which improve on the probabilistic bounds given by Proposition~\ref{prop1.4}. We improve the former in Section~\ref{anneal}, and the latter in Sections~\ref{bash} and \ref{math}.

\section{Constructing a First Subset}\label{anneal}
We constructed an $S$ with $|S| = 488$ by computer methods outlined below. This set of $488$ vectors can be found in the text file listed in Appendix~\ref{data}.
\subsection{Greedy Approach}
In order to generate their $S$ with $|S| = 480$, \cite{Cohn} used the greedy algorithm outlined in Algorithm~\ref{greedy}.
  \begin{algorithm}
    \caption{Greedy Construction of $S$}\label{greedy}
    \begin{algorithmic}[1]
      \Procedure{Greedy\_S}{$\mathcal{C}$}
      	\State $S \gets \emptyset$
        \ForAll {$v_1$ in $\mathcal{C}$} 
            \If{$\langle v_1, v_2 \rangle \leq 1$ for all $v_2 \in S \setminus \{v_1\}$}
            	\State $S \gets S \cup \{v_1\}$
            \EndIf
        \EndFor
        \State \Return {$S$}
      \EndProcedure
    \end{algorithmic}
  \end{algorithm}
In order to iterate over the elements of $\mathcal{C}$, the algorithm requires a certain ordering of the vectors. The average size of a set $S$ generated from a random order seems to hover around 238.
\vspace{-4mm}
\subsection{Simulated Annealing}
  \begin{algorithm}
    \caption{Simulated Annealing}\label{order}
    \begin{algorithmic}[1]
      \Procedure{Anneal\_S}{$\mathcal{C}$}
      	\State $S \gets \emptyset$
        \For {$1 \leq t \leq t_{max}$}
       	      \State $T \gets \textrm{temp}(t)$
       	      \If {$\textrm{rand}(0, 1) < \frac{1}{2}$ and $\exists$ $v_1 \in \mathcal{C} \setminus S$ such that $\langle v_1, x \rangle \leq 1$  $\forall x \in S$}
              	\State $c \gets $ random element of $\mathcal{C \setminus S}$ satisfying the above property
              	\State $S \gets S \cup \{c\}$
              \Else 
              	\State $s \gets $ random element of $S$
                \State $S^\prime \gets S \setminus \{s\}$
                \State $p \gets \exp( -(\textrm{energy}(S^\prime) - \textrm{energy}(S)) / T)$
                \If {$\textrm{rand(0, 1)} < p$}
              	\State $S \gets S^\prime$
              \EndIf
              \EndIf
        \EndFor
      	\State \Return{$S$}
      \EndProcedure
    \end{algorithmic}
\end{algorithm}
We improved on this approach by using the simulated annealing algorithm shown in Algorithm 2.2, which minimizes the \textit{energy} of the state $S$, whose value depends on an arbitrary energy function. We used the function $\textrm{energy}(S) = -|S|$, but any function of the current state that tends to decrease with $|S|$ could work. (For instance, a function that takes into consideration the number of vectors in $\mathcal{C} \setminus S$ that can still be added to the configuration $S \cup \{v_1\}$.)

The \textit{temperature} $\text{temp}(t)$ describes how likely the algorithm is to move to a less optimal state, and is a function of the number of iterations. In a process called the \textit{cooling schedule}, the algorithm starts at a high temperature and gradually decreases it to 0, where the program becomes exclusively greedy. Again for simplicity's sake, we used a linear cooling schedule, but other functions could also be considered.

The advantage of this algorithm is that it can explore more of the solution space without getting stuck in local extrema. Given a slow enough cooling schedule, the distribution of possible states converges over time to the global maximum for $|S|$. This algorithm led to the current record $|S| = 488$ after starting from an initial set $S$ obtained by calling the greedy algorithm on different orderings of the vectors. Due to resource constraints, we optimized based on the assumption that $S$ is antipodal (meaning that we added and removed vectors in antipodal pairs $\pm x$).

\section{Maximizing Sizes of the Subsets via Explicit Construction}\label{bash}
We also explicitly constructed a sequence of mutually disjoint subsets $S_i$ with pairwise inner product at most 1 using a probabilistic algorithm analogous to the construction of $S$. Using the algorithm developed in \cite{randomgroup}, we generated random elements $g$ of the automorphism group $\textrm{Co}_0$ and repeatedly applied them to $S$. We did this until $gS$ was disjoint from all the previous sets; if we could not find such a $gS$, we chose the largest set that we could find after deleting the overlap with previously constructed sets. After several trials and modifications, this approach yielded 51 mutually disjoint subsets of $\mathcal{C}$ of nonincreasing size, 24 of which have size 488. 
Using these subsets (which can be found as text files in Appendix~\ref{data}), we obtained the new lower bounds from Theorem~\ref{mainresult} for the kissing numbers in 25 through 31 dimensions (the bounds in the second column result from applying Proposition~\ref{prop1.4} using $|S|=488$).

\begin{thm}\label{mainresult}
The kissing numbers in 25 to 31 dimensions are bounded below by the values indicated in Table~\ref{finalbounds}.
\end{thm}
\begin{table}
\centering
\makebox[\linewidth]{
    \begin{tabular}{c | c c c}
Dimension  & $|S| = 488$ & $|S|=488$ & Previous bounds  \\
 & Explicit $S_i$'s  &  Proposition~\ref{prop1.4} & (2011) \\
\hline
25 & 197048 & 197048 & 197040\\
26 & 198512 & 198512 & 198480\\
27 & 199976 & 199968 & 199912\\
28 & 204368 & 204312 & 204188\\
29 & 208272 & 208114 & 207930\\
30 & 219984 & 219368 & 219008\\
31 & 232874 & 231412 & 230872\\
\end{tabular}}
\caption{New lower bounds for kissing numbers, based on $|S| = 488$.}
\label{finalbounds}
\vspace{-1cm}
\end{table}

Although these computer-generated subsets of $\mathcal{C}$ generate the current records, they are almost certainly suboptimal. Furthermore, they display no obvious structure or mathematical explanation; we therefore also present in Section~\ref{math} an improvement of \cite{Cohn}'s probabilistic construction of the $S_j$'s from Proposition~\ref{prop1.4}. Even using $|S| = 480$, this construction leads to improved kissing number lower bounds in dimensions 30 and 31 compared to the previous bounds.

\section{Improved Probabilistic Argument}\label{math}
The following method results in better lower bounds than the original probabilistic method used in \cite{Cohn}, although it is surpassed by the explicit constructions described in Section~\ref{bash}. It assumes that $|S| \leq 626$; the results when applied to our new $S$ of size 488 (from Section~\ref{anneal}) can be found in Table~\ref{lowerbounds}.

Recall that $\mathcal{C}$ is the set of 196560 minimal vectors of the Leech lattice $\Lambda_{24}$. Again, let $S_1 =S$, where $S$ is an antipodal subset of $\mathcal{C}$ such that any two elements $x, y \in S$ satisfy $\langle x,y \rangle \leq 1$.

For $j \geq 2$, we let $S_j$ be of the form
\[
S_j = g_jS\setminus [g_jS \cap (S_1\cup\cdots\cup S_{j-1})],
\]
where $g_j$ belongs to the automorphism group $\textrm{Co}_0$ of the Leech lattice. Since $S$ is antipodal, the $S_j$'s are also antipodal.

We recursively construct lower bounds for $|S_{k+1}|$ from previously obtained lower bounds on $|S_{i}|$, where $i \leq k$. Our argument is based on a careful examination of the set of automorphism group elements $g$ for which $|gS\cap(S_1\cup\cdots\cup S_k)|$ is less than its average value over the whole group.  

Using this method, we construct larger $S_j$'s compared to the ones obtained using the original technique in \cite{Cohn}. Using $|S| = 488$ from Section~\ref{anneal}, this leads to improved lower bounds in dimensions 28 to 31, compared to applying Proposition~\ref{prop1.4} to the same $|S|$. It also leads to improved lower bounds in dimensions 30 and 31 using $|S| = 480$ from \cite{Cohn}, compared with the lower bounds given there. However, this argument does not yield lower bounds on kissing numbers as high as those given by the explicit greedy construction from Section~\ref{bash}.

We first provide some definitions:

\begin{definition}\label{defn4.1}
\hspace{1mm}
\begin{enumerate}
\item Let $E_k = \displaystyle{\mathbb{E}_{g \in \textrm{Co}_ 0}}[|gS \cap (S_1 \cup S_2 \cup \cdots S_k)|]$, where $g$ is chosen uniformly at random from $\textrm{Co}_0$.
\item Let $\lceil x \rceil_2$ denote the smallest even integer at least $x$ and  $\lfloor x \rfloor_2$ denote the greatest even integer at most $x$.

\item For any $g \in \textrm{Co}_0$, let $\delta(g) = |gS \cap (S_1 \cup \cdots \cup S_k)| - E_k$.
\item Let $G_+ = \{ g \in \textrm{Co}_ 0 \mid \delta(g) > 0\}$ and $G_- = \{g \in \textrm{Co}_ 0 \mid \delta(g) < 0\}$.
\end{enumerate}
\end{definition}

\noindent It follows from
\[ \sum_{g \in G_+}{\delta(g)} + \sum_{g \in G_-}{\delta(g)} = \sum_{g \in G}{(|gS \cap (S_1 \cup \cdots \cup S_k)| - E_k)} = 0 \]

\noindent that
\[\sum_{g \in G_+}{|\delta(g)|} = \sum_{g \in G_-}{|\delta(g)|}.\]
We can then prove a basic property of $\delta(g)$:

\begin{lem}\label{lem4.2}
One of the following two possibilities holds:
\begin{enumerate}
\item $ \max\limits_{g \in G_-}{|\delta(g)|} = E_k - \floor{E_k}_2$, or  
\item $ \max\limits_{g \in G_-}{|\delta(g)|} \geq E_k - \floor{E_k}_2 + 2$.
\end{enumerate}
\end{lem}

\begin{proof}Due to the antipodality of $S_1, S_2, \dots, S_k$, we know $g_iv\in(S_1 \cup \cdots \cup S_k)$ if and only if $g_i(-v)\in(S_1 \cup \cdots \cup S_k)$. Therefore, $|g_iS \cap (S_1 \cup \cdots \cup S_k)|$ is an even integer. Also, $|\delta(g)| \geq E_k - \floor{E_k}_2$ for $g \in G_-$ by definition. 

In addition, note that $G_- \neq \emptyset$. This is because if $G_- = \emptyset$, then $G_+ = \emptyset$, since $\sum_{g \in G_+}{|\delta(g)|} = \sum_{g \in G_-}{|\delta(g)|}$. However, consider the identity element $e \in \textrm{Co}_0$. We have $S_1 = S = eS$, so $|eS \cap (S_1 \cup \cdots \cup S_k)| = |S| > E_k$. Therefore, $\delta(e) > 0$ and $e \in G_+$. So $G_- \neq \emptyset$. 

It follows that either $ \max\limits_{g \in G_-}{|\delta(g)|} = E_k - \floor{E_k}_2$ or $ \max\limits_{g \in G_-}{|\delta(g)|} \geq E_k - \floor{E_k}_2 + 2$.
\end{proof}

In order to analyze the implications of the two possibilities in Lemma~\ref{lem4.2}, we reformulate the problem in terms of graph theory:

\begin{definition}\label{defn4.3}
Let $H$ be a graph whose vertices correspond to the elements of $\textrm{Co}_0$. Two vertices of $H$ are adjacent if their corresponding elements $g_i$ and $g_j$ satisfy $g_iS \cap g_jS = \emptyset$.  
\end{definition}

\begin{lem}\label{lem4.4}
The graph $H$ is regular.
\end{lem}

\begin{proof}
This follows from the fact that $\textrm{Co}_0$ acts on itself by left multiplication. Consider three elements $g_1$, $g_2$, and $g'$ in $\textrm{Co}_0$. Since $\textrm{Co}_0$'s elements are automorphisms that act injectively on the Leech lattice, $g_1S \cap g'S = \emptyset$ if and only if $g_2S \cap g_2g_1^{-1}g'S = \emptyset$. This is because $g_2g_1^{-1}(g_1S \cap g'S)$ is empty if and only if $(g_1S \cap g'S)$ is empty. Then composition by $g_2g_1^{-1}$ preserves edges and therefore represents a graph automorphism of $H$. It immediately follows that the vertices corresponding to the arbitrarily chosen $g_1$ and $g_2$ have the same degree, so every vertex in $H$ has the same degree.
\end{proof}

\begin{definition}\label{defn4.5}
\hspace{1mm}
\begin{enumerate}
\item Let $P$ be the subgraph of $H$ consisting of the vertices corresponding to the elements in $G_-$. 
\item Let $R$ be the subgraph of $H$ consisting of the $k$ vertices corresponding to the elements $g_1, g_2, \dots, g_k$ associated with the subsets $S_1, \dots, S_k$. 
\item Let $Q$ be the subgraph of $H$ consisting of the rest of the vertices of $H$.
\end{enumerate}
\end{definition}

\begin{lem}\label{lem4.6}
The subgraphs $P$ and $R$ are vertex disjoint for $k \leq 50$ and $|S| \leq 626$.
\end{lem}

\begin{proof}
Recall that each vertex in $R$ corresponds to an element $g_i \in \textrm{Co}_0$, which is associated with subset $S_i$, where $i \in \{1,2,\dots, k\}$. Thus $g_iS \cap S_i = S_i$, so $|g_iS \cap (S_1 \cup \cdots \cup S_k)| \geq |S_i|$. Note that the size of each $S_i$ is greater than or equal to the size of the corresponding subset generated by applying Proposition~\ref{prop1.4}, as mentioned at the beginning of this section. Therefore, for $i \leq k \leq 50$, we have
\begin{equation*}
\begin{split}
|S_i| & \geq |S|\left( 1 - \frac{\sum_{j=1}^{i-1}{|S_j|}}{|\mathcal{C}|} \right) \\
& > |S|\left( 1 - \frac{\sum_{j=1}^{k}{|S_j|}}{|\mathcal{C}|} \right) \\
& \geq |S|\left( 1 - \frac{50|S|}{|\mathcal{C}|} \right). \\
\end{split}
\end{equation*}
Since $|S| \leq 626 < \frac{|\mathcal{C}|}{100} = 1965.6$, this is greater than
\[|S|\frac{50|S|}{|\mathcal{C}|} > |S|\frac{\sum_{j=1}^{k}{|S_j|}}{|\mathcal{C}|} = E_k. \]

Therefore $g_i$ is in $G_+$ for $i = 1,2, \dots, k$. But each vertex in $P$ corresponds to an element in $G_-$. Since $G_+$ and $G_-$ are disjoint by definition, it follows that $P$ and $R$ are disjoint.
\end{proof} 

Before we carry on analyzing the properties of $H$, we make use of the following general graph theory result:

\begin{lem}\label{lem4.7}
Let $\Omega$ be a w-regular graph. Let $A$ and $B$ be disjoint induced subgraphs of $\Omega$ such that $V(A) \cup V(B) = V(\Omega)$ and $|V(A)| \geq |V(B)|$. Then if $B$ contains an edge, $A$ must also contain an edge.
\end{lem}

\begin{proof}
Suppose, for contradiction, that there is at least one edge in $B$ but no edges in $A$. 

Let $E(A,B)$ denote the set of edges between $A$ and $B$. We will calculate the size of $E(A,B)$ in two different ways.
First, consider the sum of the degrees of the vertices in $A$. Since $\Omega$ is a $w$-regular graph, each vertex in $A$ has degree $w$, so since there are no edges in $A$, $|E(A,B)| = w|V(A)|$. 
Now consider the sum of the degrees of the vertices in $B$. Each vertex in $B$ also has degree $w$, but there is at least one edge in $B$, so $|E(A,B)| \leq (|V(B)| - 2)w + 2(w-1) = w|V(B)| - 2$. 

Therefore, 
\[w\lvert V(A)\rvert \leq w|V(B)| - 2.\] 
But $|V(A)| \geq |V(B)|$, so 
\[w|V(A)| \leq w|V(B)| - 2 \leq w|V(A)| - 2.\]
Thus we have a contradiction.
\end{proof}

We now apply Lemma~\ref{lem4.7} to our situation:

\begin{corollary}\label{cor4.8}
Suppose that $2 \leq k \leq 50$, and $|S| \leq 626$. If $|G_-| \geq |G_+|$, then there exist distinct $g_1, g_2 \in G_-$ such that $g_1S \cap g_2S = \emptyset$.
\end{corollary}

\begin{proof}
Note that by Definition~\ref{defn4.5}, $|V(P)| = |G_-|$.

Also by Definition~\ref{defn4.5}, $V(R) \cup V(P) \cup V(Q) = V(H)$. By Lemma~\ref{lem4.6}, we know that $P$ and $R$ are vertex disjoint, which together with the definition of $Q$, implies that $P$ is vertex disjoint from $R \cup Q$.  

By Proposition~\ref{prop1.4}, we know that we can construct $|S_2| \geq |S|(1-\frac{|S|}{|\mathcal{C}|})$. Thus, by the antipodality of $S_2$, we have $|S_2| = |S|$ for $|S| \leq 626$. Therefore, since $S_1$ and $S_2$ are disjoint and are both of size $|S|$, the vertices corresponding to $g_1$ and $g_2$ are adjacent. Since $g_1$ and $g_2$ are both in $R$ (as $k \geq 2$), they are also in $R \cup Q$. Thus there is an edge in $R \cup Q$.

Additionally, as proved in Lemma~\ref{lem4.4}, $H$ is a regular graph.

Then since $|G_-| \geq |G_+|$, we have $|V(P)| \geq |V(R \cup Q)|$. Thus by Lemma~\ref{lem4.7}, there must be an edge in $P$.

We let the vertices connected by this edge be $g_1$ and $g_2$, completing the proof.
\end{proof}

We remark that Lemma~\ref{lem4.6} and therefore Corollary~\ref{cor4.8} hold for values of $k$ much larger than 50; we state the condition $k \leq 50$ because, from Table~\ref{cohnsums}, we only need $k \leq 50$ to construct the kissing number lower bounds in dimensions 25 to 31.
\medskip
\\We now apply Corollary~\ref{cor4.8} to describe the implications of Lemma~\ref{lem4.2} part (1) in the next two lemmas:

\begin{lem}\label{lem4.9}
Suppose $\max\limits_{g \in G_-}{|\delta(g)|} = E_k - \floor{E_k}_2$. Then
\[|G_-| \geq \frac{|\textrm{Co}_0|}{2}(\ceil{E_k}_2 - E_k).\]
\end{lem}

\begin{proof}
Since $\max\limits_{g \in G_-}{|\delta(g)|} = E_k - \floor{E_k}_2$, we have
\[|\delta(g)| = E_k - \floor{E_k}_2\]
for all $g \in G_-$.
We also have
\[\sum_{g\in G_+}{|\delta(g)|} = \sum_{g\in G_-}{|\delta(g)|} = |G_-|(E_k - \floor{E_k}_2).\]
Now, by definition, $|\delta(g)| \geq \ceil{E_k}_2 - E_k$ for $g \in G_+$. Also, since $\max\limits_{g \in G_-}{|\delta(g)|} = E_k - \floor{E_k}_2$, by the definition of $G_-$, $E_k \neq \floor{E_k}_2$. Thus $E_k$ is not an even integer. However, $|gS \cap (S_1 \cup \cdots \cup S_k)|$ is an even integer by antipodality, so $\delta(g) \neq 0$ for all $g \in \textrm{Co}_0$. Therefore, $G_+ \cup G_- = \textrm{Co}_0$. Therefore,
\begin{equation*}
\begin{split}
|G_-|(E_k - \floor{E_k}_2) & = \sum_{g \in G_+}{|\delta(g)|} \\
& \geq (|\textrm{Co}_ 0| - |G_-|)(\ceil{E_k}_2 - E_k).
\end{split}
\end{equation*}
This implies that
\[|G_-|(E_k - \floor{E_k}_2 + \ceil{E_k}_2 - E_k) \geq |\textrm{Co}_ 0|(\ceil{E_k}_2 - E_k),\]
which simplifies to
\[|G_-| \geq \frac{|\textrm{Co}_ 0|}{2}(\ceil{E_k}_2 - E_k).\]
\end{proof}

\begin{lem}\label{lem4.10}
Let $2 \leq k \leq 50$ and $|S| \leq 626$.
Suppose that $\max\limits_{g \in G_-}{|\delta(g)|} = E_k - \floor{E_k}_2$ and $|G_-| \geq |G_+|$. Then there exist $S_{k+1}, S_{k+2} \subset \mathcal{C}$ such that 
\[|S_{k+1}|,|S_{k+2}| \geq |S| - \floor{E_k}_2,\]
and $S_\ell \cap S_m = \emptyset$ for all distinct $\ell$ and $m$ satisfying $1 \leq \ell,m \leq k+2$.
\end{lem}

\begin{proof}
We have $|G_-|$ subsets, $gS$, such that
\[E_k - |gS \cap (S_1 \cup \cdots \cup S_k)| = E_k - \floor{E_k}_2,\]
in other words
\[|gS \setminus (gS \cap (S_1 \cup \cdots \cup S_k))| = |S| - \floor{E_k}_2.\]
Now, since $|G_-| \geq |G_+|$, at least two of the $|G_-|$ subsets are disjoint by Corollary~\ref{cor4.8}. Call them $g_{\alpha}S$ and $g_{\beta}S$. Now set
\[S_{k+1} = g_{\alpha}S \setminus (g_{\alpha}S \cap (S_1 \cup \cdots \cup S_k))\] and
\[S_{k+2} = g_{\beta}S \setminus (g_{\beta}S \cap (S_1 \cup \cdots \cup S_k)).\]
\end{proof}

We consider the implications of Lemma~\ref{lem4.2} part (2):

\begin{lem}\label{lem4.11}
Suppose that $\max\limits_{g \in G_-}{|\delta(g)|} \geq E_k - \floor{E_k}_2 + 2$. Then there exists a subset $S_{k+1} \subset \mathcal{C}$ such that
\[|S_{k+1}| \geq |S| + 2 - \floor{E_k}_2\]
and $S_\ell \cap S_m = \emptyset$ for all distinct $\ell$ and $m$ satisfying $1 \leq \ell,m \leq k+1$.
\end{lem}

\begin{proof}
Suppose that $\max\limits_{g \in G_-}{|\delta(g)|} \geq E_k - \floor{E_k}_2 + 2$. Then there exists an element $g$ in $G_-$ such that 
\begin{align*}
E_k - |gS \cap (S_1 \cup \cdots \cup S_k)|  &\geq E_k - \floor{E_k}_2 + 2 \\
|gS \cap (S_1 \cup \cdots \cup S_k)| &\leq \floor{E_k}_2 - 2 \\
|gS \setminus (gS \cap (S_1 \cup \cdots \cup S_k))| &\geq |S| + 2 - \floor{E_k}_2. 
\end{align*}
Now set
\[S_{k+1} = gS \cap (S_1 \cup \cdots \cup S_k).\]
\end{proof}

Having analyzed both possibilities presented in Lemma~\ref{lem4.2}, we now derive in Lemma~\ref{lem4.12} and Corollary~\ref{cor}, a set of consequences on $|S_{k+1}|$, $|S_{k+2}|$ and $E_{k+2}$ under the assumption that $\floor{E_k}$ is even. These results will be crucial in constructing an algorithm to derive our new kissing number lower bounds.

\begin{lem}\label{lem4.12}
Let $2 \leq k \leq 50$ and $|S| \leq 626$.
If $\floor{E_k}$ is even, then either
\hspace{1mm}
\begin{enumerate}
\item There exist $S_{k+1}, S_{k+2} \subset \mathcal{C}$ such that $|S_{k+1}|,|S_{k+2}| \geq |S| - \floor{E_k}_2$ or
\item There exist $S_{k+1}, S_{k+2} \subset \mathcal{C}$ such that $|S_{k+1}| \geq |S| - \floor{E_k}_2 + 2$ and $|S_{k+2}| \geq |S| - \floor{E_k}_2 - 2$.
\end{enumerate}
\end{lem}

\begin{proof}
Let $\floor{E_k}$ be even.

By Lemma~\ref{lem4.2} either
\begin{enumerate}
\item $\max\limits_{g \in G_-}{|\delta(g)|} = E_k - \floor{E_k}_2$, or
\item $\max\limits_{g \in G_-}{|\delta(g)|} \geq E_k - \floor{E_k}_2 + 2$.
\end{enumerate}

If (1) holds, then since $\floor{E_k}$ is even, $\ceil{E_k}_2 - E_k \geq 1$, so by Lemma~\ref{lem4.9}, $|G_-| \geq \frac{|\textrm{Co}_0|}{2}$. Thus by Lemma~\ref{lem4.10}, there exist subsets $S_{k+1}$ and $S_{k+2}$ of $\mathcal{C}$ such that $|S_{k+1}|, |S_{k+2}| \geq |S| - \floor{E_k}_2$.

If (2) holds, then by Lemma~\ref{lem4.11}, there exists a subset $S_{k+1}$ of $\mathcal{C}$ such that $|S_{k+1}| \geq |S| - \floor{E_k}_2 + 2$. Then, by Proposition~\ref{prop1.4}, we can find an $S_{k+2}$ such that
\begin{equation*}
\begin{split}
|S_{k+2}| & \geq |S|\left( 1 - \frac{\sum^{k+1}_{j=1}{|S_j|}}{|\mathcal{C}|}\right) \\
& = |S|\left(1 - \frac{\sum^{k}_{j=1}{|S_j|}}{|\mathcal{C}|} - \frac{|S_{k+1}|}{|\mathcal{C}|}\right) \\
& = |S| - E_k - |S|\frac{|S_{k+1}|}{|\mathcal{C}|}.
\end{split}
\end{equation*}
Since $|S_{k+1}| \leq |S|$, for $|S| \leq 626$ we have $|S|\frac{|S_{k+1}|}{|\mathcal{C}|} \leq \frac{|S|^2}{|\mathcal{C}|}$ < 2.
Therefore, $|S_{k+2}| \geq |S| - E_k -2$,
and since $|S_{k+2}|$ is an even integer by the antipodality of the $|S|$, we have $|S_{k+2}| \geq |S| - \floor{E_k}_2 - 2$, thus proving the lemma.
\end{proof}

We remark that although we do not know whether case (1) or (2) of Lemma~\ref{lem4.2} occurs (and which of cases (1) and (2) in Lemma~\ref{lem4.12}), we show in the next corollary that the two cases imply the same consequences in terms of 
$E_{k+2}$, and suitable inequalities in linear combinations of $|S_j|$'s pertaining to Table~\ref{cohnsums} for generating kissing number lower bounds. This yields Algorithm~\ref{alg4.13} for the recursive construction of 
kissing number lower bounds.

\begin{corollary}\label{cor}
Let $|S| \leq 626$ and $2 \leq k \leq 50$.
If $\floor{E_k}$ is even, then there exist $S_{k+1}, S_{k+2} \subset \mathcal{C}$ such that
\begin{equation} \label{expk2}
E_{k+2} = |S|\frac{\sum_{j=1}^k{|S_j|} + 2(|S| - \floor{E_k}_2)}{|\mathcal{C}|},
\end{equation}
\begin{equation} \label{bound1}
|S_{k+1}| + |S_{k+2}| \geq 2(|S| - \floor{E_k}_2)
\end{equation}
and 
\begin{equation} \label{bound2}
2|S_{k+1}| + |S_{k+2}| \geq 3(|S| - \floor{E_k}_2).
\end{equation}
\end{corollary}

\begin{proof}
By Lemma~\ref{lem4.12}, if $\floor{E_k}$ is even, then there exist $S_{k+1}, S_{k+2} \in \mathcal{C}$ such that either
\begin{enumerate}
\item $|S_{k+1}|,|S_{k+2}| \geq |S| - \floor{E_k}_2$, or 
\item $|S_{k+1}| \geq |S| - \floor{E_k}_2 + 2$ and $|S_{k+2}| \geq |S| - \floor{E_k}_2 - 2$.
\end{enumerate}

If case (1) holds, then we can delete elements from $S_{k+1}$ and $S_{k+2}$ until we have $|S_{k+1}| = |S_{k+2}| = |S| - \floor{E_k}_2$. 

If case (2) holds, then we can delete elements from $S_{k+1}$ and $S_{k+2}$ until we have $|S_{k+1}| = |S| - \floor{E_k}_2 + 2$ and $|S_{k+2}| = |S| - \floor{E_k}_2 - 2$. 

Then, clearly the first and third conditions are satisfied in both cases. The second condition is also satisfied because in both cases, we have $|S_{k+1}| + |S_{k+2}| = 2(|S| - \floor{E_k}_2)$, so 
\begin{equation*}
\begin{split}
E_{k+2} & = |S|\frac{\sum_{j=1}^{k+1}{|S_j|}}{|\mathcal{C}|} \\
& = |S|\frac{\sum_{j=1}^k{|S_j|} + 2(|S| - \floor{E_k}_2)}{|\mathcal{C}|}.
\end{split}
\end{equation*}
\end{proof}

The preceding discussion allows us to use Lemma~\ref{lem4.12} to improve on the probabilistic construction (Lemmas 7.4 and 7.5) in \cite{Cohn} as follows:

\begin{alg}\label{alg4.13}
\hspace{1mm}
\begin{enumerate}
\item Set $S_1 = S$ and bound $|S_2|$ below using Proposition~\ref{prop1.4}.
\item Assume we have 
constructed $S_1, S_2, \dots, S_k$.
\item If $\floor{E_k}$ is even, Corollary~\ref{cor} ensures the existence of $S_{k+1}$ and $S_{k+2}$ such that $|S_{k+1}|$ and $|S_{k+2}|$ satisfy the bounds given by (\ref{bound1}) and (\ref{bound2}). Then $E_{k+2}$ is also given by Corollary~\ref{cor} to be as in (\ref{expk2}). Repeat from step (2), using the bounds on $S_1, \dots, S_{k+2}$.
\item If $\floor{E_k}$ is odd, then bound $|S_{k+1}|$ below using Proposition~\ref{prop1.4} and continue from step (2), using the bounds on $S_1, \dots, S_{k+1}$. 
\item Stop when $k > 50$.
\end{enumerate}
\end{alg}

We remark that Algorithm~\ref{alg4.13} is an algorithm for the recursive determination of kissing number lower bounds but it does not give specific lower bounds for the individual $|S_j|$'s, since there are two possible cases in Lemma~\ref{lem4.12} for lower bounds on $|S_{k+1}|$ and $|S_{k+2}|$.

However, Corollary~\ref{cor} and the sums shown in Table~\ref{cohnsums} for the kissing number bounds imply that assuming the lower bounds for both $|S_{k+1}|$ and $|S_{k+2}|$ to be $|S|-\lfloor E_k \rfloor_2$ (i.e., case (1) in Lemma~\ref{lem4.12}) gives at most as good a kissing lower bound and the same $E_{k+2}$ as in case (2). Since Algorithm~\ref{alg4.13} after Step (3) depends only on $E_{k+2}$, we can, for the purpose of generating valid kissing lower bounds, assume that case (1) occurs.
Therefore, we introduce Algorithm~\ref{algnew}, assuming in Step (3) of Algorithm~\ref{alg4.13} that case (1) always occurs. Algorithm~\ref{algnew} will generate the same valid kissing lower bounds as Algorithm~\ref{alg4.13}. These lower bounds must be considered only as computational devices for computing kissing lower bounds, and $S_j$'s of these sizes do not necessarily exist. Rather, Corollary~\ref{cor} gives a lower bound for  the sums of linear combinations of adjacent disjoint subsets in the sequence, and these bounds are used to give a final lower bound from the sums given in Table~\ref{cohnsums}. The bounds computed by this new algorithm give a final lower bound which is certainly at most the lower bound given by the legitimate lower bounds on the $|S_i|$'s, which is therefore correct, despite the fact that existence of subsets of sizes given by this algorithm is not guaranteed.


\begin{alg}\label{algnew}
\hspace{1mm}
\begin{enumerate}
\item Set $S_1 = S$ and bound $|S_2|$ below using Proposition~\ref{prop1.4}.
\item Assume that we have lower bounds for $|S_1|, |S_2|, \dots, |S_k|$.
\item If $\floor{E_k}$ is even, case (1) of Lemma~\ref{lem4.12} gives $|S_{k+1}|, |S_{k+2}| \geq |S| - \floor{E_k}_2$, and Corollary~\ref{cor} gives $E_{k+2}$. Repeat from step (2), using the bounds on $S_1, \dots, S_{k+2}$.
\item If $\floor{E_k}$ is odd, then bound $|S_{k+1}|$ below using Proposition~\ref{prop1.4} and continue from step (2), using the bounds on $S_1, \dots, S_{k+1}$. 
\item Stop when $k > 50$.
\end{enumerate}
\end{alg}

\begin{thm}\label{thm4.14}
Algorithm~\ref{algnew} produces the lower bounds for kissing numbers in dimensions 28 through 31 shown in Table~\ref{lowerbounds}, using $|S| = 480$ and $|S| = 488$. Although the bounds given by Algorithm~\ref{algnew} for the $|S_i|$'s are not necessarily correct, the true bounds given by Algorithm~\ref{alg4.13} give final kissing number bounds at least as high as these, which means the results in Table~\ref{lowerbounds} are true lower bounds for the kissing numbers in $\mathbb{R}^n$ for $25 \leq n \leq 31$.
\end{thm}
\begin{table}
\begin{tabular}{c|cc}
Dimension & Kissing number, $|S| = 480$ & Kissing number, $|S| = 488$\\
\hline
28 & 204188 & 204316\\
29 & 207930 & 208120\\
30 & 219012 & 219380\\
31 & 230880 & 231428\\
\end{tabular}
\caption{Lower bounds for kissing numbers produced by Algorithms ~\ref{alg4.13} and ~\ref{algnew}}
\label{lowerbounds}
\end{table}
\begin{proof}
The lower bounds for the $|S_i|$'s produced by Algorithm~\ref{algnew} using both $|S| = 488$ and $|S| = 480$ are shown in Table~\ref{bigtable}. We have also included the subset sizes produced by applying Proposition~\ref{prop1.4}, for comparison.

\begin{footnotesize}
\begin{table}
    \begin{tabular}{c|p{1.75cm}| p{1.75cm}| p{1.75cm}| p{1.75cm}}
\multirow{2}{1em}{$k$} & \multicolumn{4}{c}{Lower bound for $|S_k|$} \\
\cline{2-5} 
& \cite{Cohn}, $|S| = 480$ & Alg.~\ref{algnew}, $|S| = 480$ & \cite{Cohn}, $|S| = 488$ & Alg.~\ref{algnew}, $|S| = 488$\\
\hline
1 & 480 & 480& 488 & 488\\
2 & 480 & 480& 488 & 488\\
3 & 478 & 478& 486 & 486\\
4 & 478 & 478& 486 & 486\\
5 & 476 & 476& 484 & 484\\
6 & 476 & 476& 482 & 484\\
7 & 474 & 474& 482 & 482\\
8 & 472 & 472& 480 & 480\\
9 & 472 & 472& 480 & 480\\
10 & 470 & 470& 478 & 478\\
11 & 470 & 470& 476 & 478\\
12 & 468 & 468& 476 & 476\\
13 & 468 & 468& 474 & 474\\
14 & 466 & 466& 474 & 474\\
15 & 464 & 464& 472 & 472\\
16 & 464 & 464& 472 & 472\\
17 & 462 & 462& 470 & 470\\
18 & 462 & 462& 468 & 468\\
19 & 460 & 460& 468 & 468\\
20 & 460 & 460& 466 & 466\\
21 & 458 & 458& 466 & 466\\
22 & 456 & 458& 464 & 464\\
23 & 456 & 456& 462 & 464\\
24 & 454 & 454& 462 & 462\\
25 & 454 & 454& 460 & 460\\
26 & 452 & 452& 460 & 460\\
27 & 452 & 452& 458 & 458\\
28 & 450 & 450& 458 & 458\\
29 & 450 & 450& 456 & 456\\
30 & 448 & 448& 454 & 456\\
31 & 446 & 448& 454 & 454\\
32 & 446 & 446& 452 & 452\\
33 & 444 & 444& 452 & 452\\
34 & 444 & 444& 450 & 450\\
35 & 442 & 442& 450 & 450\\
36 & 442 & 442& 448 & 448\\
37 & 440 & 440& 448 & 448\\
38 & 440 & 440& 446 & 446\\
39 & 438 & 438& 444 & 446\\
40 & 438 & 438& 444 & 444\\
41 & 436 & 436& 442 & 442\\
42 & 436 & 436& 442 & 442\\
43 & 434 & 434& 440 & 440\\
44 & 432 & 434& 440 & 440\\
45 & 432 & 432& 438 & 438\\
46 & 430 & 430& 438 & 438\\
47 & 430 & 430& 436 & 436\\
48 & 428 & 428& 436 & 436\\
49 & 428 & 428& 434 & 434\\
50 & 426 & 426& 432 & 432\\
51 & 426 & 426& 432 & 432\\
\end{tabular}
\caption{Lower bounds on the sizes of the hypothetical $S_i$'s generated by Algorithm~\ref{algnew}}
\label{bigtable}
\end{table}
\end{footnotesize}

Using Table~\ref{cohnsums} in Section~\ref{sect1.3}, these subsets give the kissing number lower bounds shown in Table~\ref{newbounds}. Since Algorithm~\ref{algnew} generates the same kissing number lower bounds as Algorithm~\ref{alg4.13}, this proves the theorem.
\begin{table}
    \begin{tabular}{c|p{1.75cm}| p{1.75cm}| p{1.75cm}| p{1.75cm}}
Dimension & \cite{Cohn} $|S| = 480$ & Alg.~\ref{algnew}, $|S| = 480$ & \cite{Cohn}, $|S| = 488$ & Alg.~\ref{algnew}, $|S| = 488$ \\
\hline
28 & 204188 & 204188 & 204312 & 204316\\
29 & 207930 & 207930 & 208114 & 208120\\
30 & 219008 & 219012 & 219368 & 219380\\
31 & 230872 & 230880 & 231412 & 231428\\
\end{tabular}
\caption{Kissing number lower bounds}
\vspace{-0.7cm}
\label{newbounds}
\end{table}
\end{proof}

We remark that by the nature of the argument in Lemmas 7.4 and 7.5 of \cite{Cohn}, the size of the $S_i$'s decays as $i$ increases, with the rate of decay being larger for larger values of $|S|$. Therefore, Algorithm~\ref{alg4.13} and Algorithm~\ref{algnew} yield larger improvements for larger values of $|S|$. This is demonstrated by the effect of Algorithm~\ref{alg4.13} on the kissing number lower bounds for $|S| = 480$ and $|S| = 488$ shown in Table~\ref{newbounds}. We thus note that Algorithm~\ref{algnew} will yield even larger improvements if we find a larger $S$ in the future.

\section{Future Directions}
\subsection{The Structure of the Leech Lattice and its Automorphisms}

We suspect that an optimal $S$ should exhibit symmetries not currently in the $S$ of size 488. Also worth noting is the surprising fact that we could achieve 24 disjoint subsets of size 488---this suggests a deeper structure and raises the possibility that we can do much better than the greedy algorithm, with a $25^{\text{th}}$ or even $51^\text{st}$ disjoint copy. Even changing the greedy algorithm to a backtracking one might increase the number of disjoint subsets we can fit, although the running time might become an issue.

\subsection{Graph-theoretic approaches}
We might translate the problem to a maximum independent set problem or a maximum clique problem and approach the problem graph-theoretically. These problems are generally NP-complete, and we would still need clever techniques adapted to symmetric graphs, approximation algorithms, or other graph-theoretic heuristics to make this approach fruitful. 
\subsection{Improving current computational methods}
The algorithms we used could be optimized in a number of ways. New search heuristics may result in a larger $S$, as may a simply longer simulated annealing with a slower temperature decrease using more computational power and time. The explicit construction of $S_i$ can almost certainly be improved. Given that we used a greedy algorithm, we strongly doubt that the $S_i$'s are anywhere near optimal; a clever modification to how the program traverses the search tree will most likely improve the result.
\vspace{-0.3cm}
\begin{appendices}

\section{Data Files}\label{data}
The data files associated with the explicit constructions described in this paper can be found in the Git repository at \url{https://github.com/kenzkallal/Kissing-Numbers}. The files of the form \texttt{S\_i.txt} contain the lists of vectors in the subsets $S_i$. We also include the file \texttt{minvects.txt}, containing our ordering of the vectors of $\mathcal{C}$, as well as \texttt{Vbasis.txt}, which is the basis we used for $\Lambda_{24}$.
\end{appendices}
\section*{Acknowledgements}
The results in this paper originated from a research lab project in PROMYS 2015. We are deeply grateful to Henry Cohn for proposing this problem, and for his mentoring. We also thank our PROMYS research lab counselor Claudia Feng for her support, as well as Erick Knight, David Fried, Glenn Stevens, the PROMYS Foundation, and the Clay Mathematics Institute for making this research possible. 

\bibliography{citations.bib}
\bibliographystyle{amsalpha}
\end{document}